\documentclass{article}

\usepackage[all]{xy}
\usepackage{mathrsfs}
\usepackage{amssymb}
\usepackage{amsthm}
\usepackage{amsmath}
\usepackage{amsfonts}
\usepackage{authblk}

\usepackage{url}

\newtheorem{thm}{Theorem}
\newtheorem*{thm*}{Theorem}

\newtheorem{prop}[thm]{Proposition}
\newtheorem{cor}[thm]{Corollary}

\newtheorem{defn}{Definition}

\allowdisplaybreaks

\author[1]{Yaacov Kopeliovich}
\author[2]{Camilo Sanabria Malagón\thanks{corresponding author: c.sanabria135@uniandes.edu.co}}

\affil[1]{School of Business\\ University of Connecticut\\ Storrs, CT, USA}
\affil[2]{Departamento de Matemáticas\\ Universidad de los Andes\\ Bogotá, Colombia}

\title{Schwarz maps for modular curves}

\begin{document}

\date{}

\maketitle

\begin{abstract}

  We solve a classical problem posed by F. Klein and studied by A. Hurwitz concerning the construction of linear ordinary differential equations associated with modular transformations of fixed degree. For every odd integer $N\ge 3$ (respectively, even integer $N\ge 4$), we construct a canonical invariant model of the modular curve $X(N)=\mathbb{H}/\Gamma(N)$ (respectively, $X_H(N)=\mathbb{H}/H(N)$ where $H(N)=\Gamma(N)\cap\Gamma_0^0(2N)$), together with a linear ordinary differential equation with rational coefficients whose Schwarz map parametrizes this model and whose projective monodromy group is the finite quotient $PSL_2(\mathbb{Z})/\tilde{\Gamma}(N)$ (respectively, $PSL_2(\mathbb{Z})/\tilde{H}(N)$). The construction is expressed in terms of invariant projective geometry and Picard-Vessiot theory and yields equations that are canonical up to projective equivalence. In this framework, Hurwitz's classical equation for degree $7$ appears as a special case of a general mechanism. The results place Klein's question within the modern theory of algebraic linear ordinary differential equations and provide a uniform geometric realization of modular transformation groups as projective differential Galois groups. As an application, we construct an explicit example of a linear ordinary differential equation associated with $X(9)$.
\end{abstract}

\section*{Introduction}

  In an 1885 paper, A. Hurwitz studied a question posed by F. Klein concerning the construction of linear ordinary differential equations associated with modular transformations of fixed degree \cite{HURWITZ1886}. In the case of degree $7$, Hurwitz explicitly constructed a third-order Fuchsian linear ordinary differential equation whose solutions realize the modular transformation, thereby producing an equation with prescribed finite projective monodromy. This construction, closely related to independent work of G. H. Halphen \cite{HALPHEN1884}, may be regarded as an early instance of the problem of realizing a finite group as the projective monodromy group of a linear ordinary differential equation.

  Klein's interest in this problem was part of his broader Erlangen program \cite{KLEIN1893}, and in particular, of interpreting analytic objects through their transformation groups. Modular equations of degree $N$ (i.e., the algebraic relations between $j(\tau)$ and $j(N\tau)$) carry natural actions of finite groups arising from quotients of the modular group, and Klein sought analytic realizations of these symmetries analogous to his geometric realizations of the same groups as automorphism groups of algebraic curves, such as the modular curve of level $7$ \cite{KLEIN1878}. Indeed, earlier on, he identified the modular curve of level $5$ with the icosahedron and found an associated hypergeometric differential equation \cite{KLEIN1878b}. Constructing a linear differential equation whose solutions transform according to a prescribed modular symmetry amounts to realizing the same group action at the level of analytic continuation and monodromy. In modern terms, this asks for a linear ordinary differential equation whose Schwarz map parametrizes a model of the corresponding modular curve and whose projective monodromy group coincides with the finite quotient of the modular group under consideration.

  From a modern perspective, Klein and Hurwitz's work anticipates several notions that now play a central role in the study of linear ordinary differential equations with finite differential Galois groups. In particular, the ratios of solutions define what is now called a Schwarz map, whose image is a projective curve invariant under a finite linear group and whose monodromy is the projectivization of the differential Galois group. The appearance of invariant polynomials, quotient constructions, and algebraic relations among solutions indicates that the problem is naturally situated within Picard-Vessiot theory and invariant theory, even though these frameworks were not yet fully available at the time.

  Hurwitz cautiously observed that his method might not extend directly to modular transformations of higher degree. Beyond the exceptional case of degree $7$, the Fuchs relation admits multiple compatible collections of characteristic exponents, and the classical approach provides no intrinsic criterion for selecting a distinguished differential equation. This ambiguity reflects the fact that the problem is global and geometric in nature, while the classical method relies primarily on local analytic data.

  Subsequent developments have clarified important aspects of this situation. Beukers and Heckman classified hypergeometric equations of type ${}_nF_{n-1}$ with algebraic solutions by determining their monodromy representations through the interlacing criterion \cite{BEUKERS1989}. Van der Put and Ulmer analyzed linear differential equations with finite differential Galois groups and three singularities, developing methods to constrain admissible local exponent data \cite{VANDERPUT2000}. These results provide structural information and explicit normal forms in important special cases. However, they do not resolve Klein's original question: neither the restriction of exponent data nor the classification of hypergeometric equations produces a canonical differential equation attached to a modular transformation of arbitrary degree.

  The guiding principle of the present paper is that Klein's question should be formulated in terms of Schwarz maps and invariant projective geometry, rather than in terms of characteristic exponents. Following the approach developed in \cite{SANABRIA2017}, we view an algebraic linear ordinary differential equation through the geometry of its Schwarz map and the associated quotient Schwarz map. In this setting, linear ordinary differential equations with finite differential Galois groups are naturally organized by invariant curves in projective space together with their quotients by finite linear groups.

  The main contribution of this paper is to demonstrate that, for modular transformations of arbitrary degree, the associated invariant projective curve can be constructed explicitly and that there exists a linear ordinary differential equation whose Schwarz map parametrizes a model of this curve. The resulting equation is canonical up to projective equivalence. The ambiguity present in Hurwitz's approach is absorbed into the geometry of the quotient and into the choice of a suitable Picard-Vessiot extension, rather than appearing as an indeterminacy in the local exponent data.

  Another limitation anticipated by Hurwitz is related to his reliance on the Abel-Jacobi map, which embeds the modular curve into a projective space of dimension equal to the genus of the curve. He observed that this ambient dimension is excessive in the sense that the resulting differential equation admits further reduction. A key tool in our construction is the use of theta functions with characteristics, as developed by Farkas, Kra, and Kopeliovich \cite{FARKAS1996}, which provide explicit models of modular curves in projective space together with projective representations of the corresponding quotient of the modular group. These models yield invariant projective curves whose symmetry groups coincide with the projective monodromy groups under consideration, thereby allowing the construction of linear ordinary differential equations with prescribed Schwarz maps and finite differential Galois groups.

  Within this framework, Hurwitz's construction for degree $7$ appears as a special case of a general mechanism. More generally, the differential equations associated to modular transformations of higher degree are characterized by their Schwarz maps and invariant geometry, rather than by ad hoc choices of characteristic exponents. This places Klein's original question within a setting where it admits a natural and systematic solution consistent with the modern theory of algebraic linear ordinary differential equations.

  \begin{thm*}
  For every odd integer $N\ge 3$, let $X(N)$ denote the modular curve associated with the principal congruence subgroup $\Gamma(N)$; and for every even integer $N\ge 4$, let $X_H(N)$ denote the modular curve associated with the subgroup $H(N)=\Gamma(N)\cap\Gamma_0^0(2N)$. There exists a canonical invariant projective model $C_N$ of $X(N)$ (respectively, $X_H(N)$), together with a linear ordinary differential equation with rational coefficients whose Schwarz map parametrizes $C_N$ and whose projective monodromy group is isomorphic to the finite quotient $PSL_2(\mathbb{Z})/\tilde{\Gamma}(N)$ (respectively, $PSL_2(\mathbb{Z})/\tilde{H}(N)$). The equation is unique up to projective equivalence, and its Picard-Vessiot extension is an Abelian extension of the function field of $X(N)$ (respectively, $X_H(N)$).
  \end{thm*}

  The present paper is devoted to the construction of these equations and is organized as follows. Section 1 is dedicated to the preliminary geometrical and analytical framework; we recall the necessary properties of theta functions with characteristics and employ them to define an explicit invariant projective model for the modular curves of level $N$, followed by a review of the relevant structure of orbit spaces and Schwarz maps. In Section 2, we establish the core theoretical results of the paper, generalizing principles from algebraic Picard-Vessiot theory to central extensions, and we prove the existence of the desired canonical linear ordinary differential equations with prescribed projective monodromy and its universal property. Finally, in Section 3, as a concrete application of this general mechanism, we conclude by explicitly constructing a fourth-order linear ordinary differential equation associated with the modular curve of level $9$.

\section{Preliminaries}

  Let $\mathbb{H}$ denote the upper half-plane and let $\tau\in\mathbb{H}$. The group $\Gamma(1)=SL_2(\mathbb{Z})$, which is generated by $\gamma_S=\begin{pmatrix} 0 & 1 \\ -1 & 0 \end{pmatrix}$ and $\gamma_T=\begin{pmatrix} 1 & 1 \\ 0 & 1 \end{pmatrix}$, acts on $\mathbb{H}$ by fractional linear transformations
  \[
     \gamma\tau=\frac{a\tau+b}{c\tau+d},
     \qquad 
     \gamma=\begin{pmatrix} a & b \\ c & d \end{pmatrix}.
  \]
  Given an integer $N\geq 3$, we denote by $\Gamma(N)$ the principal congruence subgroup of level $N$ (i.e., $a,d\equiv 1\pmod N$, $b,c\equiv 0\pmod N$), and by $X(N)=\mathbb{H}/\Gamma(N)$ the compact modular curve associated with $\Gamma(N)$. Similarly, we denote by $H(N)$ the subgroup $\Gamma(N)\cap\Gamma_0^0(2N)$ (i.e., $a,d\equiv 1\pmod N$, $b,c\equiv 0\pmod {2N}$), and by $X_H(N)=\mathbb{H}/H(N)$ the compact modular curve associated with $H(N)$.

\subsection{Theta functions with characteristics}\label{thetasec}

  We recall the definition of theta functions with characteristics and the transformation properties required to construct our geometric models. A detailed account of this theory is given in \cite{FARKAS2001}.

  For $a,b\in\mathbb{R}$, we define the theta function with characteristic $[a,b]$ by
  \begin{align*}
  \theta\!\begin{bmatrix} a \\ b \end{bmatrix}(z,\tau) & 
  =
  \sum_{n\in\mathbb{Z}}
  \exp\!\Big(
      \pi i (n+a)^2 \tau
      + 2\pi i (n+a)(z+b)
  \Big) \\
   & =e^{\pi i a^2 \tau + 2\pi i a(z+b)}\vartheta(z+a\tau+b,\tau),
  \end{align*}
  where $z\in\mathbb{C}$ and $\vartheta(z,\tau)$ is the classical theta function with characteristic $[0,0]$. In particular, given two characteristics $[a_0,b_0]$ and $[a_1,b_1]$, we have the relation
  \[
  \theta\!\begin{bmatrix} a_1 \\ b_1 \end{bmatrix}(z,\tau) = e^{\pi i (\Delta a)^2 \tau + 2\pi i \Delta a(z+ b_1)} \theta\!\begin{bmatrix} a_0 \\ b_0 \end{bmatrix}(z+\Delta a\tau+\Delta b,\tau),
  \]
  where $\Delta a = a_1-a_0$ and $\Delta b = b_1-b_0$. Note that our choice of $a$ and $b$ in the series corresponds to $a/2$ and $b/2$ under the conventions of \cite{FARKAS2001}.

  We shall also require the following standard identities:
  \begin{align*}
    \theta\!\begin{bmatrix} a+m \\ b+n \end{bmatrix}(z,\tau) & = e^{2\pi i a n} \theta\!\begin{bmatrix} a \\ b \end{bmatrix}(z,\tau) \quad \text{for } m,n\in\mathbb{Z}, \quad \text{and} \\
    \theta\!\begin{bmatrix} -a \\ -b \end{bmatrix}(z,\tau) & = \theta\!\begin{bmatrix} a \\ b \end{bmatrix}(-z,\tau).
  \end{align*}

  Theta functions with characteristics are a well-known family of holomorphic functions used to construct explicit functions on elliptic curves. Indeed, if we fix $\tau\in\mathbb{H}$, the function $\theta\!\begin{bmatrix} a \\ b \end{bmatrix}(z,\tau)$ is a holomorphic function of $z$ that satisfies the quasi-periodicity relations
  \begin{align*}
     \theta\!\begin{bmatrix} a \\ b \end{bmatrix}(z+1,\tau) &= \theta\!\begin{bmatrix} a \\ b \end{bmatrix}(z,\tau), \\
     \theta\!\begin{bmatrix} a \\ b \end{bmatrix}(z+\tau,\tau) &= e^{-\pi i \tau - 2\pi i (a+b)} \theta\!\begin{bmatrix} a \\ b \end{bmatrix}(z,\tau).
  \end{align*}
  In particular, the quotient of two theta functions with characteristics $[a,b]$ and $[a',b']$ satisfying $a+b\equiv a'+b'\pmod 1$ defines a meromorphic function on the elliptic curve $E_\tau=\mathbb{C}/(\mathbb{Z}+\tau\mathbb{Z})$. The quotient
  $$\theta\!\begin{bmatrix} a \\ b \end{bmatrix}(z,\tau)\Bigg/ \theta\!\begin{bmatrix} a' \\ b' \end{bmatrix}(z,\tau)$$
  has a simple zero at the point $(1/2-a)\tau+(1/2-b)$ and a simple pole at $(1/2-a')\tau+(1/2-b')$, provided these two points do not coincide. More generally, by virtue of these quasi-periodicity relations, theta functions with characteristics define holomorphic sections of line bundles over $E_\tau$.

  Viewed as functions of $\tau\in\mathbb{H}$, they satisfy the following transformation laws under the action of $SL_2(\mathbb{Z})$:
  \begin{align*}
    \theta\!\begin{bmatrix} a \\ b \end{bmatrix}(z,\tau+1) & = e^{-\pi i a(a+1)} \theta\!\begin{bmatrix} a \\ b+a-1/2 \end{bmatrix}(z,\tau), \\
    \theta\!\begin{bmatrix} a \\ b \end{bmatrix}\left(\frac{z}{\tau},-\frac{1}{\tau}\right) & = \sqrt{-i\tau} e^{\pi i z^2/\tau - 2\pi i ab} \theta\!\begin{bmatrix} -b \\ a \end{bmatrix}(z,\tau).
  \end{align*}
  These transformation properties allow us to construct explicit models of modular curves and to analyze the corresponding actions of the modular group. To this end, we consider the values of these functions at $z=0$.

  \begin{defn}\label{thetadef}
  Given $a,b\in\mathbb{Q}$, the value at $z=0$,
  \[
     \theta_{a,b}(\tau)=\theta\!\begin{bmatrix} a \\ b \end{bmatrix}(0,\tau),
  \]
  is called a theta constant (or \emph{Thetanullwert}) with characteristic $[a,b]$.
  \end{defn}

  Our model of $X(N)$ relies on the decomposition of the theta series into residue classes modulo $N$. More precisely, for $a,b\in\mathbb{Q}$, we have
  \begin{align*}
    \theta_{a,b}(\tau) & = \sum_{l=0}^{N-1} \theta_{\frac{l+a}{N},Nb}(N^2\tau).
  \end{align*}

  We now specialize to the specific families of theta constants required for our constructions (cf. \cite[Chapter 3]{FARKAS2001}).

  \subsection*{Odd $N$ case}

  Let $N\geq 3$ be an odd integer. We consider the $M=\frac{N-1}{2}$ theta constants of level $N$ defined by
  \[
  \phi_l(\tau)=\theta_{\frac{2l+1}{2N},\frac{1}{2}}(N\tau),\qquad l=0,1,\ldots,\frac{N-3}{2}.
  \]
  These constants span a finite-dimensional vector space that is stable under the natural action of the modular group.

  The transformation laws of the theta functions with characteristics induce a projective action on the space spanned by the ordered components of the level $N$ theta constants. Specifically, if we fix the ordering
  \[
  \Phi_N(\tau)=\left[\phi_0(\tau),\ldots,\phi_{(N-3)/2}(\tau)\right],
  \]
  it transforms under $\gamma_S$ and $\gamma_T$ according to
  \[
    \Phi_N(\gamma_S\tau)= \kappa(\gamma_S,\tau)\,\Phi_N(\tau)\, S_*
    \quad \text{and} \quad
    \Phi_N(\gamma_T\tau)= \kappa(\gamma_T,\tau)\,\Phi_N(\tau)\, T_*,
  \]
  where $\kappa(\gamma,\tau)$ is a scalar automorphy factor, and the matrices $S_*$ and $T_*$ are given by
  \begin{align*}
    S_*= & \left[e^{2l(k+1)2\pi i/2N}+e^{-(2k(l+1)+1)2\pi i/2N}\right]_{l,k=0}^{(N-3)/2},\\
    T_*= & \mathrm{diag}\!\left(e^{(l^2+l) 2\pi i/2N}\right)_{l=0}^{(N-3)/2}.
  \end{align*}

  The matrix $S_*$ has order $2$, $T_*$ has order $N$, and the map $\rho_N : SL_2(\mathbb{Z}) \longrightarrow PGL_M(\mathbb{C})$ defined by $\rho_N(\gamma_S)=S_*$ and $\rho_N(\gamma_T)=T_*$ yields a projective representation whose projective kernel is $\pm\Gamma(N)$. Letting $\tilde{\Gamma}(N)=\pm\Gamma(N)/\pm I$, the map $\rho_N$ factors through a faithful projective representation of the finite quotient group $PSL_2(\mathbb{Z}/N\mathbb{Z})\simeq PSL_2(\mathbb{Z})/\tilde{\Gamma}(N)$.

  \subsection*{Even $N$ case}

  Let $N\geq 4$ be an even integer. We consider the $M=\frac{N}{2}+1$ theta constants of level $N$ defined by
  \[
  \phi_l(\tau)=\theta_{\frac{2l}{2N},0}(N\tau),\qquad l=0,1,\ldots,\frac{N}{2}.
  \]
  As in the odd case, these constants span a finite-dimensional vector space stable under the natural action of the modular group. Under the fixed ordering
  \[
  \Phi_N(\tau)=\left[\phi_0(\tau),\ldots,\phi_{N/2}(\tau)\right],
  \]
  the vector transforms under $\gamma_S$ and $\gamma_T$ via
  \[
    \Phi_N(\gamma_S\tau)= \kappa(\gamma_S,\tau)\,\Phi_N(\tau)\, S_*
    \quad \text{and} \quad
    \Phi_N(\gamma_T\tau)= \kappa(\gamma_T,\tau)\,\Phi_N(\tau)\, T_*,
  \]
  where $\kappa(\gamma,\tau)$ is a scalar automorphy factor and
  \begin{align*}
    S_*= & \left[\epsilon_k\left(e^{2lk 2\pi i/2N}+e^{-2lk 2\pi i/2N}\right)\right]_{l,k=0}^{N/2},\\
    T_*= & \mathrm{diag}\!\left(e^{l^2 2\pi i/2N}\right)_{l=0}^{N/2},
  \end{align*}
  with $\epsilon_k=1$ if $k\ne 0,N/2$ and $\epsilon_k=1/2$ otherwise.
  In this framework, $S_*$ has order $2$ and $T_*$ has order $2N$. The map $\rho_N : SL_2(\mathbb{Z}) \longrightarrow PGL_M(\mathbb{C})$ defined by $\rho_N(\gamma_S)=S_*$ and $\rho_N(\gamma_T)=T_*$ yields a projective representation whose projective kernel is $\pm H(N)$. It consequently factors through a faithful projective representation of the finite group $PSL_2(\mathbb{Z})/\tilde{H}(N)$, where $\tilde{H}(N)=\pm H(N)/\pm I$.

\subsection{A model for the modular curve of level $N$}\label{modelsec}

  The faithful projective representation constructed above establishes the geometric framework for the uniform structural results required in the sequel.

  \begin{thm}\label{uniformthm}
  The theta constants with characteristics of level $N\ge 3$ define a holomorphic map
  \[
     \Phi_N : \mathbb{H} \longrightarrow \mathbb{P}^{M-1}(\mathbb{C}),
  \]
  where $M=(N-1)/2$ if $N$ is odd (respectively, $M=N/2+1$ if $N$ is even). This map is invariant under $\Gamma(N)$ when $N$ is odd, and induces an embedding
  \[
     X(N) \hookrightarrow \mathbb{P}^{M-1}(\mathbb{C}) \qquad (\text{respectively, } X_H(N) \hookrightarrow \mathbb{P}^{M-1}(\mathbb{C}))
  \]
  which is equivariant with respect to the natural action of $PSL_2(\mathbb{Z}/N\mathbb{Z})$ on $X(N)$ and the projective representation induced by $\rho_N$ on $\mathbb{P}^{M-1}(\mathbb{C})$ (respectively, the natural action of $PSL_2(\mathbb{Z})/\tilde{H}(N)$ on $X_H(N)$ and the projective representation induced by $\rho_N$ on $\mathbb{P}^{M-1}(\mathbb{C})$).
  \end{thm}

  In particular, the image $C_N$ of $\Phi_N$ is a model of $X(N)$ (respectively, $X_H(N)$), realized as a projective algebraic curve whose homogeneous coordinates are given by theta constants. The action of $SL_2(\mathbb{Z})$ on $\mathbb{H}$ induces, via $\rho_N$, an action of the finite group $PSL_2(\mathbb{Z}/N\mathbb{Z})$ (respectively, $PSL_2(\mathbb{Z})/\tilde{H}(N)$) on the model $C_N$.

  Consequently, the theta constants of level $N$ simultaneously yield:
  \begin{itemize}
  \item[(i)] a projective model $C_N$ of the modular curve $X(N)$ (respectively, $X_H(N)$);
  \item[(ii)] a projective representation of the finite modular group $PSL_2(\mathbb{Z}/N\mathbb{Z})$ (respectively, $PSL_2(\mathbb{Z})/\tilde{H}(N)$);
  \item[(iii)] an explicit system of coordinates in which the group action is realized by projective transformations.
  \end{itemize}

  The curve $C_N$, together with the projective representation $\rho_N$, serves as the geometric input enabling the construction, via Schwarz maps and algebraic Picard--Vessiot theory, of linear differential equations whose projective monodromy coincides with these prescribed modular transformation groups.

  \subsection*{Algebraic relations among theta constants}

  The algebraic relations among theta constants determine the defining equations of $C_N$ in projective space. More precisely, the homogeneous ideal of the model curve $C_N\subset\mathbb{P}^{M-1}(\mathbb{C})$ is generated by the homogeneous algebraic relations satisfied by these theta constants.
  
  For odd $N$, the homogeneous algebraic relations among the theta constants of level $N$ can be systematically obtained via the residue theorem and are described explicitly in \cite[Theorem 4.8]{QUINE1997}.
  
  Fix $\tau\in\mathbb{H}$. For $r=1,2\ldots,\frac{N-1}{2}$, let
  \[
  \psi_r(z)=\theta\!\begin{bmatrix} \frac{N-2r}{2N} \\ \frac{1}{2} \end{bmatrix}(z,N\tau),
  \]
  so that $\psi_r(z)$ has a zero at $z_r=(1/2-(N-2r)/2N)N\tau+(1/2-1/2)=r\tau$ and satisfies $\psi_r(0)=\phi_{\frac{N-1}{2}-r}(\tau)$. Moreover, given an integer $s$ with $1\le s\le (N-1)/2$, we have
  \[
  \psi_r(z_s)=
  \begin{cases}
    e^{-\pi i\frac{s^2\tau+s}{N}}\phi_{\frac{N-1}{2}-(r-s)}(\tau), & s<r, \\[8pt]
    -e^{-\pi i\frac{s^2\tau+s}{N}}e^{-2\pi i \frac{s-r}{N}}\phi_{\frac{N-1}{2}-(s-r)}(\tau), & s>r,
  \end{cases}
  \]
  and $\psi'_s(z_s)=e^{-\pi i\frac{s^2\tau+s}{N}}\theta_{\frac{1}{2},\frac{1}{2}}(N\tau)$. Now, given a collection of integers
  \[
  l_1,\ldots,l_m \in \left\{1,2,\ldots,\dfrac{N-1}{2}\right\}
  \]
  and a collection of distinct integers
  \[
  s_1,\ldots,s_m \in \left\{1,2,\ldots,\dfrac{N-1}{2}\right\}
  \]
  such that $\{l_1,\ldots,l_m\} \cap \{s_1,\ldots,s_m\} = \emptyset$ and
  \[
  \sum_{j=1}^m s_j \equiv \sum_{j=1}^m l_j \pmod N,
  \]
  the product
  \[
  F(z)=\prod_{j=1}^m \frac{\psi_{l_j}(z)}{\psi_{s_j}(z)}
  \]
  defines a meromorphic function on the elliptic curve $E_{N\tau}$ with simple poles at the points $z_{s_j}$. By the residue theorem, the sum of the residues of $F(z)$ at its poles vanishes, yielding the rational relation
  \[
  \sum_{j=1}^{m}\dfrac{\prod_{i} \psi_{l_i}(z_{s_j})}{\prod_{i\ne j} \psi_{s_i}(z_{s_j})}=0.
  \]
  Clearing denominators and rewriting the expression in terms of theta constants, these identities yield homogeneous algebraic dependencies among the projective coordinates
  \[
  [\theta_0(\tau):\theta_1(\tau):\ldots:\theta_{N-1}(\tau)],
  \]
  thereby defining elements in the homogeneous ideal of the model curve $C_N\subset\mathbb{P}^{M-1}(\mathbb{C})$.

  For small odd values of $N$, the ideal of the model curve $C_N$ is generated by these relations (see \cite[Chapter 3]{FARKAS2001}). For general $N$, and specifically for even $N$, the homogeneous relations among the theta constants of level $N$ require a more detailed analysis (cf. \cite[Chapter 10]{MUMFORD2007} and \cite{MUMFORD1966}).

  To prepare for the construction of differential equations with prescribed projective monodromy, we must analyze the geometry of the orbit spaces arising from actions of finite linear groups. We therefore recall some foundational facts concerning these orbit spaces, Schwarz maps, and the relevant results from algebraic Picard--Vessiot theory.

\subsection{Space of orbits}\label{orbitspace}

  We recall the basic structure of orbit spaces arising from linear actions of finite groups; a detailed treatment may be found in \cite{BRION2010}.

  Let $G\subseteq GL_n(\mathbb{C})$ be a finite group and let $R=\mathbb{C}[X_1,\ldots,X_n]$ be the coordinate ring of $\mathbb{C}^n$. The natural right action of $G$ on $R$ is given by
  \[
  g: X_j \mapsto \sum_{i=1}^n X_i g_{ij},
  \qquad g=(g_{ij})\in G.
  \]
  This induces a left action on $\mathbb{C}^n$ via
  \[
  g:(x_1,\ldots,x_n)\mapsto \Big(\sum_{i=1}^n x_i g^{-1}_{i1},\ldots,\sum_{i=1}^n x_i g^{-1}_{in}\Big),
  \]
  where $g^{-1}=(g^{-1}_{ij})$ denotes the inverse of $g$.

  Since $G$ is finite, it is reductive, and the invariant ring $R^G$ separates $G$-orbits. The affine orbit space $\mathbb{C}^n/G$ is therefore identified with $\mathrm{Spec}(R^G)$.

  Let $F_1,\ldots,F_r$ be homogeneous generators of $R^G$. The quotient map
  \[
  \mathbf{x} \longmapsto \big(F_1(\mathbf{x}),\ldots,F_r(\mathbf{x})\big)
  \]
  embeds $\mathbb{C}^n/G$ as an affine algebraic subvariety of $\mathbb{C}^r$. There exists a Zariski-open dense subset of $\mathbb{C}^n$ on which this map has maximal rank.

  The same construction applies in projective space. The group $G$ acts on $\mathbb{P}^{n-1}(\mathbb{C})$ and on its homogeneous coordinate ring. Let $\mathbb{P}^{n-1}(\mathbb{C})/G$ denote the corresponding orbit space and let
  \[
  \Pi:\mathbb{P}^{n-1}(\mathbb{C})\longrightarrow \mathbb{P}^{n-1}(\mathbb{C})/G
  \]
  be the canonical quotient map.

  Let $\Lambda\in\mathbb{Z}_{>0}$ be an integer such that the space $R^G_\Lambda$ of homogeneous invariants of degree $\Lambda$ defines a projective coordinate system for the quotient, i.e.,
  \[
  \mathbb{P}(R^G_\Lambda)\simeq \mathbb{P}^{n-1}(\mathbb{C})/G.
  \]
  If $\Phi_1,\ldots,\Phi_s$ is a basis of $R^G_\Lambda$, the map
  \[
  [\mathbf{x}] \longmapsto \big[\Phi_1(\mathbf{x}):\cdots:\Phi_s(\mathbf{x})\big]
  \]
  embeds $\mathbb{P}^{n-1}(\mathbb{C})/G$ into $\mathbb{P}^{s-1}(\mathbb{C})$. As in the affine setting, this map achieves maximal rank on a dense open subset.

\subsection{Schwarz maps}

  Let $C_0$ be a compact Riemann surface and $K=\mathbb{C}(C_0)$ its field of meromorphic functions. Let $\delta:K\to K$ be a non-trivial derivation.

  Consider a linear ordinary differential equation
  \[
  L(y)=\delta^n(y)+a_{n-1}\delta^{n-1}(y)+\cdots+a_0 y=0,
  \qquad a_i\in K,
  \]
  and denote by $S_0\subset C_0$ its set of singular points.

  Given a non-singular point $p\in C_0\setminus S_0$ and a basis of local solutions $\mathbf{y}=(y_1,\ldots,y_n)$ defined on a neighborhood $U$ of $p$, the \emph{Schwarz map} is defined as the analytic continuation of
  \[
  [\mathbf{y}]: U \longrightarrow \mathbb{P}^{n-1}(\mathbb{C}),
  \qquad z \longmapsto [y_1(z):\cdots:y_n(z)].
  \]

  The projective monodromy group $PG\subset PGL_n(\mathbb{C})$ is the image of the monodromy representation associated with $\mathbf{y}$. When $PG$ is finite, composition with the quotient map yields a rational map
  \[
  \psi: C_0 \dashrightarrow \mathbb{P}^{n-1}(\mathbb{C})/PG,
  \]
  defined away from $S_0$, which we call the \emph{quotient Schwarz map}.

\subsection{Algebraic Picard--Vessiot theory}

  We recall the aspects of Picard--Vessiot theory relevant to the algebraic case; see \cite{VANDERPUT2003} for a foundational reference.

  Assume that $L(y)=0$ is irreducible and that all its solutions are algebraic over $K$. If $(y_1,\ldots,y_n)$ is a basis of solutions, the Picard--Vessiot extension is given by
  \[
  F = K[y_1,\ldots,y_n].
  \]

  Let $I$ be the kernel of the $K$-morphism
  \[
  \Phi: K[Y_1,\ldots,Y_n] \longrightarrow F,
  \qquad Y_j \mapsto y_j.
  \]
  The differential Galois group is identified with the finite subgroup $G\subset GL_n(\mathbb{C})$ consisting of elements $g=(g_{ij})$ preserving $I$ under the action
  \[
  Y_j \longmapsto \sum_{i=1}^n Y_i g_{ij}.
  \]

  If $P\in K[Y_1,\ldots,Y_n]$ is $G$-invariant, then $\Phi(P)\in K$. Since $L(y)=0$ is irreducible, the group $G$ is reductive.

  \begin{thm}\label{compointthm}
  Let $L(y)=0$ be an irreducible linear ordinary differential equation with finite differential Galois group $G$. Then the ideal $I$ is generated by its $G$-invariant elements. In particular, if $P_1,\ldots,P_r$ generate $\mathbb{C}[Y_1,\ldots,Y_n]^G$, then
  \[
  I=\langle P_1-f_1,\ldots,P_r-f_r\rangle,
  \]
  where $f_i=\Phi(P_i)$.
  \end{thm}

\subsection{Schwarz maps for invariant curves}\label{schwarzinvsec}

  We now specialize the previous constructions to invariant curves.  
  Let $G\subset GL_n(\mathbb{C})$ be a finite group acting on $\mathbb{P}^{n-1}(\mathbb{C})$, and let
  \[
  C \subset \mathbb{P}^{n-1}(\mathbb{C})
  \]
  be an algebraic, $G$-invariant, irreducible curve not contained in any projective hyperplane. Let $\nu:C_0\to C/G$ be a normalization of the quotient curve and set $K=\mathbb{C}(C_0)$.

  The following result shows that invariant curves arise naturally as images of Schwarz maps of linear differential equations.

  \begin{thm}[{\cite[Theorem 2]{SANABRIA2024}}]\label{theoabelianext}
  Let $G$ be a finite algebraic subgroup of $GL_n(\mathbb{C})$ and let 
  $C\subseteq\mathbb{P}^{n-1}(\mathbb{C})$ be an algebraic, $G$-invariant, irreducible curve not contained in a projective hyperplane. Let $\nu:C_0\rightarrow C/G$ be a normalization and set $K=\mathbb{C}(C_0)$. Then there exist:
  \begin{itemize}
  \item[(i)] a branched covering $\pi:C_1\rightarrow C_0$ such that 
  $E=\mathbb{C}(C_1)$ is an abelian extension of $K$;

  \item[(ii)] a linear differential equation of order $n$ with coefficients in $E$
  whose Schwarz map parametrizes $C$, and whose quotient Schwarz map is
  $\psi=\nu\circ\pi$.
  \end{itemize}
  \end{thm}

  The theorem provides a conceptual mechanism to construct differential equations with prescribed projective monodromy from invariant curves. The idea of the proof is that the invariant relations
  \[
  F_i(\mathbf{y})=f_i ,\qquad i=1,\ldots,r,
  \]
  may be repeatedly differentiated. This produces rational expressions in $y_1,\ldots,y_n$ for $\mathbf{y}$ and its first $n$ derivatives. The linear dependence relation between $\mathbf{y}, \mathbf{y}', \ldots, \mathbf{y}^{(n)}$ is invariant under the action of $G$; therefore, its coefficients lie in the base field $E$. In particular, once the quotient data and the invariant functions are known, the coefficients of the differential equation can be computed explicitly. This leads to the following algorithmic formulation introduced in \cite{SANABRIA2024}.

\subsection*{Algorithm}

  Assume that $K\subseteq\overline{\mathbb{Q}(z)}$ is a computable differential field, and that
  \[
  F_1,\ldots,F_r \in \mathbb{C}[X_1,\ldots,X_n]^G
  \]
  are homogeneous generators defined over $\mathbb{K}=K\cap\overline{\mathbb{Q}}$, indexed so that
  \[
  \det\!\left(\frac{\partial(F_1,\ldots,F_n)}{\partial(X_1,\ldots,X_n)}\right)\neq 0 .
  \]
  We take $\delta=d/dz$.

  \medskip
  \noindent\textbf{Input:} $f_1,\ldots,f_r\in K$ such that  
  \begin{itemize}
  \item[(i)] $I=\langle F_1-f_1,\ldots,F_r-f_r\rangle$ is a maximal ideal of 
  $K[X_1,\ldots,X_n]$;
  \item[(ii)] the associated projective curve $C\subset\mathbb{P}^{n-1}(\mathbb{C})$ is not contained in a projective hyperplane.
  \end{itemize}

  \noindent\textbf{Output:} coefficients $(a_{n-1},\ldots,a_0)\in K^n$ such that
  \[
  \delta^n(y)+a_{n-1}\delta^{n-1}(y)+\cdots+a_0 y=0
  \]
  has a fundamental system of solutions whose Schwarz map parametrizes $C$.

  \begin{enumerate}
  \item Fix a monomial ordering in $X_1,\ldots,X_n$ and compute a Gröbner basis 
  $B$ of $\langle F_1-f_1,\ldots,F_r-f_r\rangle$.

  \item For $i=1,\ldots,n$, set $Y_{0,i}=X_i$ and denote 
  $\mathbf{Y}_0=(Y_{0,1},\ldots,Y_{0,n})^{\intercal}$.

  \item Let $\mathbf{J}^{-1}$ be the remainder modulo $B$ of an inverse (modulo $I$) of the Jacobian matrix
  \[
  \frac{\partial(F_1,\ldots,F_n)}{\partial(X_1,\ldots,X_n)}.
  \]

  \item Define $\mathbf{Y}_1$ as the remainder modulo $B$ of 
  $\mathbf{J}^{-1}\mathbf{f}$, where $\mathbf{f}=(f_1,\ldots,f_n)^{\intercal}$.

  \item Extend $\delta$ to $K(X_1,\ldots,X_n)$ by setting 
  $\delta X_i = Y_{1,i}$.

  \item For $j=2,\ldots,n$, recursively define $\mathbf{Y}_j$ by taking
  $\mathbf{Y}_j$ to be the remainder modulo $B$ of $\delta \mathbf{Y}_{j-1}$.

  \item Let $\mathbf{Y}^{-1}$ be the remainder modulo $B$ of an inverse (modulo $I$) of the matrix
  \[
  [\mathbf{Y}_{n-1}\,|\,\cdots\,|\,\mathbf{Y}_1\,|\,\mathbf{Y}_0].
  \]

  \item Return the remainder modulo $B$ of
  \[
  -\mathbf{Y}^{-1}\mathbf{Y}_n,
  \]
  which yields the coefficient vector $(a_{n-1},\ldots,a_0)$.
  \end{enumerate}

  By Compoint's theorem (Theorem \ref{compointthm}), every irreducible linear differential equation with coefficients in a computable field and finite differential Galois group arises from this procedure.

  \medskip

  In practical computations, one typically works inside a Computer Algebra System over the base field $\mathbb{Q}(z)$. When the required coefficient field $K$ is a computable extension contained in $\overline{\mathbb{Q}(z)}$, the algorithm is implemented by adjoining generators of $K$ to $\mathbb{Q}(z)$ together with their minimal polynomials. Concretely, one enlarges the polynomial ring $\mathbb{Q}(z)[X_1,\ldots,X_n]$ by introducing new variables corresponding to the algebraic generators of $K$, and simultaneously enlarges the ideal $\langle F_1-f_1,\ldots,F_r-f_r\rangle$ by adjoining the minimal polynomials of these generators. Gröbner bases are then computed in this extended ring. This procedure allows the entire computation to be carried out effectively over $\mathbb{Q}(z)$ while representing coefficients in the desired algebraic extension.

  The previous construction produces linear ordinary differential equations whose Schwarz maps parametrize invariant curves and whose projective monodromy is controlled by the action of a finite group, in the sense that the projective monodromy is isomorphic to a subgroup of the image of the finite group in the projective linear group. A difference between the projective monodromy group of the LODE produced by the algorithm and the image of the finite group in the projective linear group may arise when the functions $f_i$ in the algorithm do not belong to the field $K$ but to $E$, using the notation of the theorem and the algorithm. Thus, in order to apply this framework to the inverse problem considered by Klein and obtain exactly the prescribed group $PSL_2(\mathbb{Z})/\tilde{\Gamma}(N)$ or $PSL_2(\mathbb{Z})/\tilde{H}(N)$ as projective monodromy, it will be necessary to understand the relationship between the action of the finite group, the projective monodromy group, and the full differential Galois group of the resulting equation.

\section{Central extensions and differential Galois groups}

  We state and prove a generalization of Theorem \ref{theoabelianext} which eliminates the need to pass to an abelian extension of the base field. Furthermore, the differential Galois group of the resulting equation will be a unimodular central extension of a finite projective linear group. The idea behind the proof relies on an argument of van der Put (cf. Lemma 3.3 in \cite{VANDERPUT2020}).

  We first introduce some notation. Let $G\subset GL_n(\mathbb{C})$ be a finite algebraic subgroup and let $PG$ be the image of $G$ in $PSL_n(\mathbb{C})$. We denote by $PG^{SL}$ the smallest group in $SL_n(\mathbb{C})$ whose image in $PSL_n(\mathbb{C})$ is $PG$. The group $PG^{SL}$ is a central extension of $PG$ by a finite cyclic group. Let $m$ denote the order of the cover $PG^{SL}\rightarrow PG$. 

  As in Theorem \ref{theoabelianext}, let $C\subseteq\mathbb{P}^{n-1}(\mathbb{C})$ be an algebraic, $G$-invariant, irreducible curve not contained in any projective hyperplane, and let $\nu:C_0\rightarrow C/G$ be a normalization. Set $K=\mathbb{C}(C_0)$. We will assume from now on that $PG$ acts faithfully on $C/G$, and therefore that $PG$ is isomorphic to the Galois group of the extension $K\le\mathbb{C}(C)$.
  
  Let $\mathcal{O}(1)$ be the line bundle of degree one on $\mathbb{P}^{n-1}(\mathbb{C})$ and let $\mathcal{L}=\imath^*\mathcal{O}(1)$, where $\imath:C\rightarrow \mathbb{P}^{n-1}(\mathbb{C})$ is the inclusion. The space of global sections of $\mathcal{L}$ is isomorphic to $\mathbb{C}^n$ because the image of $C$ is not contained in a projective hyperplane. Therefore, we obtain a right action of $PG^{SL}$ on the space of global sections of $\mathcal{L}$ that is isomorphic to the action on $\mathbb{C}[X_1,\ldots,X_n]$. We denote by $\mathbf{y}=(y_1,\ldots,y_n)$ a basis of global sections of $\mathcal{L}$ that transforms according to this representation.

  Let $V=\langle y_1,\ldots,y_n\rangle_\mathbb{C}$ be the vector space of global sections of $\mathcal{L}$. We want to construct a cyclic cover $\tilde{C}\rightarrow C$ such that the vector space $\tilde{V}=f V/y_1$ lies in $\mathbb{C}(\tilde{C})$ and is $PG^{SL}$-equivariant to $V$ for some $f\in\mathbb{C}(\tilde{C})$. Note that the analytic continuation of the map 
  \[
  z \mapsto \left[f(z): f(z)\frac{y_2}{y_1}(z):\ldots:f(z)\frac{y_n}{y_1}(z)\right]
  \]
  also yields a Schwarz map that parametrizes $C$, provided that $\tilde{V}$ is the space of solutions of a linear ordinary differential equation (as will be established in the proof of Theorem \ref{theo0} below).

  The existence of a cyclic cover $\tilde{C}\rightarrow C$ satisfying the desired properties is a consequence of Hilbert's Theorem 90. Indeed, $\mathbb{C}(C)$ is a Galois extension of $\mathbb{C}(C/G)\simeq K$ with Galois group $PG$, and if $\tilde{\sigma}$ denotes any lifting of $\sigma\in PG$ to $PG^{SL}$, then the map $\sigma\mapsto \left(\tilde{\sigma}(y_1)/y_1\right)^m$ is a 1-cocycle of $PG$ with values in $\mathbb{C}(C)^\times$, since
  \[
  \left(\dfrac{\tilde{\sigma}(y_1)}{y_1}\right)^m\cdot\sigma\left(\left(\dfrac{\tilde{\tau}(y_1)}{y_1}\right)^m\right) = \left(\dfrac{\tilde{\sigma}(y_1)}{y_1}\cdot\dfrac{\tilde{\sigma}\tilde{\tau}(y_1)}{\tilde{\sigma}(y_1)}\right)^m = \left(\dfrac{\widetilde{\sigma\tau}(y_1)}{y_1}\right)^m.
  \]
  Therefore, Hilbert's Theorem 90 guarantees the existence of $f_0\in\mathbb{C}(C)^\times$ such that
  \[
  \left(\dfrac{\tilde{\sigma}(y_1)}{y_1}\right)^m=\dfrac{\sigma(f_0)}{f_0}
  \]
  for every $\sigma\in PG$. The desired cyclic cover $\tilde{C}\rightarrow C$ corresponds to the extension $\mathbb{C}(\tilde{C})=\mathbb{C}(C)(f)$ where $f^m=f_0$. To establish this, first, we must verify that $\mathbb{C}(C)(f)$ is a Galois extension over $\mathbb{C}(C)$ of degree $m$; and second, that the vector space $\tilde{V}=f V/y_1$ is $PG^{SL}$-equivariant to $V$.

  To establish the first claim, it suffices to show that $X^m-f_0$ is irreducible over $\mathbb{C}(C)$. If $X^m-f_0$ were reducible, there would exist a proper divisor $d$ of $m$ such that $X^d-f_0$ has a root $g_0$ in $\mathbb{C}(C)$. In that case, we would have $\left(\tilde{\sigma}(y_1)/y_1\right)^m=\left(\sigma(g_0)/g_0\right)^d$ for any lifting $\tilde{\sigma}\in PG^{SL}$ of $\sigma\in PG$. Consider the subgroup $H$ of $PG^{SL}$ composed of the $m/d$ liftings $\tilde{\sigma}$ of each $\sigma\in PG$ such that $\left(\tilde{\sigma}(y_1)/y_1\right)^{m/d}=\sigma(g_0)/g_0$. The group $H$ defines a subgroup of $SL_n(\mathbb{C})$ whose image in $PSL_n(\mathbb{C})$ is also $PG$, strictly contained in $PG^{SL}$, which contradicts the minimality of the latter.

  Regarding the second claim, for every lifting $\tilde{\sigma}\in PG^{SL}$ of $\sigma\in PG$, from the equality $\left(\tilde{\sigma}(y_1)/y_1\right)^m=\sigma(f^m)/f^m$ we define
  \[
  \tilde{\sigma}(f)=f\dfrac{\tilde{\sigma}(y_1)}{y_1},
  \]
  so that the map $V\rightarrow \tilde{V}$ given by $v\mapsto fv/y_1$ is an isomorphism of representations.
  
  In particular, the extension $\mathbb{C}(C)(f)$ is Galois over $K$ with Galois group isomorphic to $PG^{SL}$, and therefore $\text{Gal}(\mathbb{C}(C)(f)/K)$ is an extension of $PG$ by the cyclic group $\text{Gal}(\mathbb{C}(C)(f)/\mathbb{C}(C))$ of order $m$.

  We synthesize the previous construction in the following theorem.

  \begin{thm}\label{theo0}
    Let $G$ be a finite algebraic subgroup of $GL_n(\mathbb{C})$ and $C\subseteq\mathbb{P}^{n-1}(\mathbb{C})$ be an algebraic, $G$-invariant, irreducible curve not contained in any projective hyperplane. Let $\nu:C_0\rightarrow C/G$ be a normalization and set $K=\mathbb{C}(C_0)$. If $PG$, the image of $G$ in $PSL_n(\mathbb{C})$, acts faithfully on $C$, then there exists a linear differential equation of order $n$ with coefficients in $K$ whose Schwarz map parametrizes $C$, and whose quotient Schwarz map is $\nu$. Moreover, the equation can be taken so that its differential Galois group is $PG^{SL}$, the minimal central extension of $PG$ in $SL_n(\mathbb{C})$.
  \end{thm}

  \begin{proof}
    We retain the notation from the previous discussion. The field $\mathbb{C}(\tilde{C})=K[\tilde{V}]$ is a Galois extension of $K$ with Galois group $PG^{SL}$. Furthermore, $PG^{SL}$ acts linearly on $\tilde{V}$, meaning $\tilde{V}$ is the space of solutions of a linear ordinary differential equation with coefficients in $K$ and differential Galois group $PG^{SL}$. The analytic extension of the map $z\mapsto [f(z): f(z)\frac{y_2}{y_1}(z):\ldots:f(z)\frac{y_n}{y_1}(z)]$ is a Schwarz map that parametrizes $C$, and the corresponding quotient Schwarz map is $\nu$.
  \end{proof}

  We briefly recall the concept of projective equivalence. Two linear ordinary differential equations defined over the same differential field are \emph{projectively equivalent over $K$} if there exists a function $h$ such that $h'/h\in K$ and $hy$ is a solution of one equation whenever $y$ is a solution of the other. In this case, one can choose bases of solutions for both equations such that their associated Schwarz maps coincide. Note that if $f=h'/h$, then $h=\exp\left(\int\! f\right)$.

  \begin{cor}
    The linear ordinary differential equation obtained in Theorem \ref{theo0} is uniquely determined up to projective equivalence by the curve $C$ and the normalization $\nu:C_0\rightarrow C/G$.
  \end{cor}

  \begin{proof}
    We keep the notation from above. Let $F_1,\ldots,F_r \in \mathbb{C}[X_1,\ldots,X_n]^{PG^{SL}}$ be a set of homogeneous generators and $d_i=\deg(F_i)$ for $i=1,\ldots,r$. By Galois correspondence and the fact that $PG^{SL}$ is the Galois group of $\mathbb{C}(\tilde{C})$ over $K$, we have that $f_i(z)=F_i\left(f(z), f(z)\frac{y_2}{y_1}(z),\ldots,f(z)\frac{y_n}{y_1}(z)\right)$ belongs to $K$ for each $i=1,\ldots,r$, and
    \[
    \nu^*\left(\dfrac{F_i^{l_i}}{F_j^{l_j}}\right) = \dfrac{f_i^{l_i}}{f_j^{l_j}}
    \]
    for any $i,j\in\{1,\ldots,r\}$ and $l_i,l_j\in\mathbb{Z}_{>0}$ such that $l_id_i=l_jd_j$, provided $C/G\not\subseteq\{F_j=0\}$. To consider another normalization, let $g_1,\ldots,g_r\in K$ be such that $g_i^{l_i}/g_j^{l_j}=f_i^{l_i}/f_j^{l_j}$ for any $i,j\in\{1,\ldots,r\}$ and $l_i,l_j\in\mathbb{Z}_{>0}$ such that $l_id_i=l_jd_j$. If we define $h$ such that $h^{d_i} = g_i/f_i=h_i$, then we obtain $h^{l_jd_j}=h^{l_id_i}=(g_i/f_i)^{l_i}=(g_j/f_j)^{l_j}$ and $h'/h=\frac{h'_i}{d_ih_i}\in K$. Furthermore,
    \[
    F_i\left(h(z)f(z), h(z)f(z)\frac{y_2}{y_1}(z),\ldots,h(z)f(z)\frac{y_n}{y_1}(z)\right)=h^{d_i}(z)f_i(z)=g_i(z),
    \]
    and therefore the linear ordinary differential equation determined by the maximal ideal $\langle F_1-g_1,\ldots,F_r-g_r\rangle$ is projectively equivalent to the one determined by $\langle F_1-f_1,\ldots,F_r-f_r\rangle$.
  \end{proof}

  \begin{cor}
    Under the assumptions and notation of Theorem \ref{theo0}, let $L_{PG,C}(y)=0$ be the linear ordinary differential equation obtained. Let $H$ be a normal subgroup of $PG$, and let $L_{H,C}(y)=0$ be the linear ordinary differential equation obtained by applying Theorem \ref{theo0} to the action of $H$ on the curve $C$ and a normalization $\mu: C_1 \rightarrow C/H$. Then, $L_{H,C}(y)=0$ is projectively equivalent to the pullback of $L_{PG,C}(y)=0$ by a rational map $p: C_1 \rightarrow C_0$.
  \end{cor}

  \begin{proof}
    Let $\pi_H: C/H \rightarrow C/PG=C/G$ be the projection induced by the inclusion $H \subset PG$. Since $\nu: C_0 \rightarrow C/G$ and $\mu: C_1 \rightarrow C/H$ are normalizations, there exists a rational map $p: C_1 \rightarrow C_0$ such that $\nu\circ p = \pi_H\circ\mu$.
    \[
    \xymatrix{
      C_1\ar[rr]^{\mu}\ar[d]_{p} & & C/H\ar[d]^{\pi_H}\\
      C_0\ar[rr]_{\nu} & & C/G
    }
    \]
    Thus, if $(y_{PG,1},\ldots,y_{PG,n})$ and $(y_{H,1},\ldots,y_{H,n})$ are bases of solutions of $L_{PG,C}(y)=0$ and $L_{H,C}(y)=0$ respectively, such that their associated Schwarz map parametrizes $C$, and if $F_1(X_1,\ldots,X_n)$ and $F_2(X_1,\ldots,X_n)$ are homogeneous elements in $\mathbb{C}[X_1,\ldots,X_n]$ invariant under $PG^{SL}$ with $F_2$ not vanishing identically on $C$, then
    \begin{align*}
    \dfrac{F_1}{F_2}\left(y_{G,1}\circ p,\ldots,y_{G,n}\circ p\right) & = (\nu \circ p)^*\left(\dfrac{F_1}{F_2}\right) = (\pi_H\circ\mu)^*\left(\dfrac{F_1}{F_2}\right) \\
     & = \mu^*\circ\pi_H^*\left(\dfrac{F_1}{F_2}\right) = \mu^*\left(\dfrac{F_1}{F_2}\right)\\
     & = \dfrac{F_1}{F_2}(y_{H,1},\ldots,y_{H,n}).
    \end{align*}
    Hence, the quotient Schwarz map $\mu$ is a lifting to $C/H$ of the composition $\nu\circ p$, and therefore there exists a $g\in G$ such that $[y_{PG,1}\circ p:\ldots:y_{PG,n}\circ p]g=[y_{H,1}:\ldots:y_{H,n}]$. The quotient Schwarz map of the pullback by $p$ of $L_{PG,C}(y)=0$ coincides with the normalization $\mu$. From the first corollary of Theorem \ref{theo0}, we conclude that $L_{H,C}(y)=0$ is projectively equivalent to the pullback of $L_{PG,C}(y)=0$ by $p$.
  \end{proof}

  If we apply the algorithm from the previous section to obtain the equation from Theorem \ref{theo0}, using the group $PG^{SL}$ in place of $G$, the Galois correspondence implies that the functions $f_i$ belong to $K$. Therefore, the resulting equation will be defined over $K$ instead of over an abelian extension of $K$ as in Theorem \ref{theoabelianext}. The equation obtained in this way satisfies the following universal property, first identified by Klein for the equations associated with the tetrahedral modular curve $X(3)$, the octahedral modular curve $X(4)$, and the icosahedral modular curve $X(5)$ \cite{KLEIN187711,KLEIN187712}.

  \begin{prop}
    Let $G\subset GL_n(\mathbb{C})$ be a finite algebraic subgroup and let $C\subseteq\mathbb{P}^{n-1}(\mathbb{C})$ be an algebraic, $G$-invariant, irreducible curve not contained in any projective hyperplane. Let $\nu:C_0\rightarrow C/G$ be a normalization and set $K=\mathbb{C}(C_0)$. Assume that $PG$, the image of $G$ in $PSL_n(\mathbb{C})$, acts faithfully on $C$. Let $L_C(y)=0$ be the linear ordinary differential equation from Theorem \ref{theo0} with Galois group $PG^{SL}$.
    
    Let $L(y)=0$ be a linear ordinary differential equation of order $n$ defined over a field $K_L=\mathbb{C}(C_L)$, where $C_L$ is a normal algebraic curve. Assume that the Schwarz map of $L(y)=0$ parametrizes $C$ for some choice of basis of solutions. If the projective monodromy for this choice is a normal subgroup of $PG$, then there exists a rational map $p: C_L \rightarrow C_0$ such that $L(y)=0$ is projectively equivalent to the pullback of $L_C(y)=0$ by $p$; i.e., there exists a function $f\in K_L$ such that
    \[
    y=e^{\int\! f}\cdot(y_0 \circ p)
    \]
    is a solution of $L(y)=0$ whenever $y_0$ is a solution of $L_C(y_0)=0$.
  \end{prop}  

  \begin{proof}
    Let $H$ denote the projective monodromy group of $L(y)=0$ for the given choice of basis of solutions $(y_1,\ldots,y_n)$. We take this basis of solutions so that it parametrizes $C$. By assumption, $H$ is a normal subgroup of $PG$, therefore $PG$ acts on $C/H$, and the quotient map $C\rightarrow C/G$ factors through $C\rightarrow C/H$. As above, let $\pi_H: C/H \rightarrow C/G$ be the induced projection. If $\psi: C_L \dashrightarrow C/H$ is the quotient Schwarz map associated with $L(y)=0$, then $\pi_H\circ\psi: C_L \rightarrow C/G$ is a branched cover. Since $\nu: C_0 \rightarrow C/G$ is a normalization, there exists a rational map $p: C_L \rightarrow C_0$ such that $\nu\circ p = \pi_H\circ\psi$.
    \[
    \xymatrix{
      C_L\ar@{-->}[rr]^{\psi}\ar[d]_{p} & & C/H\ar[d]^{\pi_H}\\
      C_0\ar[rr]_{\nu} & & C/G
    }
    \]

    We denote by $\tilde{H}$ the representation of the differential Galois group of $L(y)=0$ in $GL_n(\mathbb{C})$ for the given choice of basis of solutions. Let $(y_{C,1},\ldots,y_{C,n})$ be a basis of solutions of $L_C(y)=0$ such that the associated Schwarz map parametrizes $C$.

    The proof proceeds in three parts. First, we show that the proposition holds if $\tilde{H}$ is $PG^{SL}$; then, we show that it holds if $H$ is $PG$; and finally, we establish it in full generality, where $H$ is any normal subgroup of $PG$.
    
    Assume first that $\tilde{H}=PG^{SL}$. In this case, $H=PG$ and therefore $\pi_H$ is the identity, making $\nu\circ p = \psi$ the quotient Schwarz map of $L(y)=0$. The second corollary of Theorem \ref{theo0} implies that $L(y)=0$ is projectively equivalent to the pullback by $p$ of $L_C(y)=0$, and therefore the proposition holds.

    We now assume that $H=PG$. It is a well-known fact that we can replace $L(y)=0$ with a projectively equivalent equation such that the representation of the differential Galois group lies in $SL_n(\mathbb{C})$. Thus, we may assume $\tilde{H}\subseteq SL_n(\mathbb{C})$. In that case, $\tilde{H}$ is the product of $PG^{SL}$ and a finite cyclic group $C_l$ of order $l$. Indeed, every $\gamma \in \tilde{H}$ is the product of an element in $PG^{SL}$ and a scalar matrix in $SL_n(\mathbb{C})$. Let $E_L$ be the Picard--Vessiot extension of $L(y)=0$ and $\tilde{K}_L$ the fixed field $E_L^{PG^{SL}}$, where $PG^{SL}$ is identified with a subgroup of $\tilde{H}$. Furthermore, let $\tilde{C}_L$ be a compact Riemann surface with field of meromorphic functions $\mathbb{C}(\tilde{C}_L)\simeq \tilde{K}_L$. In particular, the Galois group of $E_L$ over $\tilde{K}_L$ is $PG^{SL}$, and by the first part of the proof, there exists a rational map $\tilde{p}: \tilde{C}_L \rightarrow C_0$ such that $L(y)=0$ is projectively equivalent to the pullback by $\tilde{p}$ of $L_C(y)=0$. This means that if $q: \tilde{C}_L \rightarrow C_L$ is the projection induced by the inclusion $K_L\subseteq \tilde{K}_L$, then there exists a function $f\in\tilde{K}_L$ such that
  \[
    y_i\circ q=e^{\int\! f}\cdot\left(y_{C,i}\circ \tilde{p}\right), \quad \text{for } i=1,\ldots,n.
  \]
  \[
  \xymatrix{
       & & & & C\ar[dd]\\
      \tilde{C}_L\ar[rrd]^{\tilde{p}}\ar[d]_{q}\ar@{..>}[rrrru]^{[\mathbf{y}]\circ q} & & & &\\
      C_L\ar[rr]_{p}\ar@{..>}[rrrruu]^{[\mathbf{y}]} & & C_0\ar@{..>}[rruu]_{[\mathbf{y}_C]}\ar[rr]_{\nu} & & C/G
  }
  \]
  We will show that $\tilde{p}$ factors through $q$ and that $f$ is in $K_L$. For this, as before, let $F_1,\ldots,F_r \in \mathbb{C}[X_1,\ldots,X_n]^{PG^{SL}}$ be a set of homogeneous generators of degrees $d_1,\ldots,d_r$, and let $f_1,\ldots,f_r\in K$ be defined by $f_i=F_i\left(y_{C,1},\ldots,y_{C,n}\right)$ for each $i=1,\ldots,r$. Thus, we have $\tilde{g}_1,\ldots,\tilde{g}_r\in \tilde{K}_L$ given by
  \begin{align*}
    \tilde{g}_i &=F_i\left(y_1\circ q,\ldots,y_n\circ q\right)=F_i\left(y_1,\ldots,y_n\right)\circ q\\
      &=F_i\left(e^{\int\! f}\cdot(y_{C,1}\circ \tilde{p}),\ldots,e^{\int\! f}\cdot(y_{C,n}\circ \tilde{p})\right)\\
      &=e^{d_i\int\! f}\cdot F_i\left(y_{C,1}\circ \tilde{p},\ldots,y_{C,n}\circ \tilde{p}\right)\\
      &=e^{d_i\int\! f}\cdot \left(f_i\circ \tilde{p}\right).
  \end{align*}
  Now if $l \mid l_1d_1+\ldots+l_rd_r$, then $F_1^{l_1}\cdots F_r^{l_r}$ and $\tilde{g}_1^{l_1}\cdots \tilde{g}_r^{l_r}$ are invariant under the action of $C_l$, meaning the latter belongs to $K_L$. Therefore, if we also have $m_1d_1+\ldots+m_rd_r=l_1d_1+\ldots+l_rd_r$ and $\tilde{g}_1^{m_1}\cdots \tilde{g}_r^{m_r}\ne 0$, then
\begin{align*}
  \nu^*\left(\dfrac{F_1^{l_1}\cdots F_r^{l_r}}{F_1^{m_1}\cdots F_r^{m_r}}\right) &= \dfrac{f_1^{l_1}\cdots f_r^{l_r}}{f_1^{m_1}\cdots f_r^{m_r}},\quad \text{and}\\
  (\nu\circ\tilde{p})^*\left(\dfrac{F_1^{l_1}\cdots F_r^{l_r}}{F_1^{m_1}\cdots F_r^{m_r}}\right) &= \dfrac{\tilde{g}_1^{l_1}\cdots \tilde{g}_r^{l_r}}{\tilde{g}_1^{m_1}\cdots \tilde{g}_r^{m_r}}\\
   &= (\nu\circ p\circ q)^*\left(\dfrac{F_1^{l_1}\cdots F_r^{l_r}}{F_1^{m_1}\cdots F_r^{m_r}}\right),
\end{align*}
  and therefore $p\circ q=\tilde{p}$. Because $\tilde{p}$ factors through $q$, we deduce that $C_l$ acts on $e^{d_i\int\! f}=\tilde{g}_i/\left(f_i\circ p\circ q\right)$ by scalars for all $i=1,\ldots,r$. Thus, a power of $e^{\int\! f}$ lies in $K_L$. This implies that $f\in K_L$, since $\left(e^{m\int\! f}\right)'/e^{m\int\! f}=mf$ for every integer $m$. Hence, $y_i=e^{\int\! f}\cdot\left(y_{C,i}\circ p\right)$ for $i=1,\ldots,n$, and $L(y)=0$ is projectively equivalent to the pullback by $p$ of $L_C(y)=0$.

    Finally, assume that $H$ is a normal subgroup of $PG$. Let $L_{C,H}(y)=0$ be the linear ordinary differential equation obtained by applying Theorem \ref{theo0} to the action of $H$ on $C$ and a normalization $\mu: C_1 \rightarrow C/H$. By the second corollary of Theorem \ref{theo0}, $L_{C,H}(y)=0$ is projectively equivalent to the pullback of $L_C(y)=0$ by a rational map $p_H: C_1 \rightarrow C_0$. Furthermore, based on the second part of the proof established above, $L(y)=0$ is projectively equivalent to the pullback of $L_{C,H}(y)=0$ by a rational map $p_H': C_L \rightarrow C_1$. Consequently, $L(y)=0$ is projectively equivalent to the pullback of $L_C(y)=0$ by the composite rational map $p=p_H\circ p_H': C_L \rightarrow C_0$.
  \end{proof}

\section{Schwarz maps for modular curves}

  We now direct our attention to the applicability of the preceding results, in particular Theorem \ref{theo0}, to the case of modular curves.
  
  Let $\Gamma$ be a finite-index subgroup of $PSL_2(\mathbb{Z})$ and let $X_\Gamma=\mathbb{H}/\Gamma$ be the associated modular curve. The quotient map $\pi: X_\Gamma \rightarrow X(1)\simeq X_\Gamma/PSL_2(\mathbb{Z})$ is a branched cover, and we consider a normalization $\nu: \mathbb{P}^1(\mathbb{C}) \rightarrow X_\Gamma/PSL_2(\mathbb{Z})$. The group $PSL_2(\mathbb{Z})/\Gamma$ acts faithfully on $X_\Gamma$. Thus, to apply Theorem \ref{theo0}, we must find, on the one hand, a model $C_\Gamma$ of $X_\Gamma$ in a projective space $\mathbb{P}^{n-1}(\mathbb{C})$, together with a projective representation of $PSL_2(\mathbb{Z})/\Gamma$ in $PGL_n(\mathbb{C})$, making the induced map $X_\Gamma \rightarrow C_\Gamma\subseteq \mathbb{P}^{n-1}(\mathbb{C})$ equivariant; and, on the other hand, a lifting of this projective representation to a linear representation in $GL_n(\mathbb{C})$ of a central extension of $PSL_2(\mathbb{Z})/\Gamma$.
  
  Our discussion in Section \ref{thetasec} on theta functions with characteristics provides a systematic way to obtain the desired equivariant model for the case where $\Gamma=\tilde{\Gamma}(N)$ and $N\ge 3$ is odd, or where $\Gamma=\tilde{H}(N)$ and $N\ge 4$ is even. We now state and prove our main theorem.

  \begin{thm}\label{mainthm}
  For every odd integer $N\ge 3$, let $X(N)$ denote the modular curve associated with the principal congruence subgroup $\Gamma(N)$; and for every even integer $N\ge 4$, let $X_H(N)$ denote the modular curve associated with the subgroup $H(N)=\Gamma(N)\cap\Gamma_0^0(2N)$. There exists a canonical invariant projective model $C_N$ of $X(N)$ (respectively, $X_H(N)$), together with a linear ordinary differential equation with rational coefficients whose Schwarz map parametrizes $C_N$ and whose projective monodromy group is isomorphic to the finite quotient $PSL_2(\mathbb{Z})/\tilde{\Gamma}(N)$ (respectively, $PSL_2(\mathbb{Z})/\tilde{H}(N)$). The equation is unique up to projective equivalence, and its Picard-Vessiot extension is an Abelian extension of the function field of $X(N)$ (respectively, $X_H(N)$).
  \end{thm}

  \begin{proof}
    The existence of the canonical invariant projective model $C_N$ of $X(N)$ (respectively, $X_H(N)$) is a consequence of Theorem \ref{uniformthm}. The existence of the linear ordinary differential equation with rational coefficients whose Schwarz map parametrizes $C_N$ and whose projective monodromy group is isomorphic to the finite quotient $PSL_2(\mathbb{Z})/\tilde{\Gamma}(N)$ (respectively, $PSL_2(\mathbb{Z})/\tilde{H}(N)$) is a consequence of Theorem \ref{theo0} once we take a normalization of the quotient of $C_N$, which is a rational curve. The uniqueness up to projective equivalence follows from the first corollary of Theorem \ref{theo0}. Finally, the differential Galois group of the equation from Theorem \ref{theo0} is a central extension of the finite quotient $PSL_2(\mathbb{Z})/\tilde{\Gamma}(N)$ (respectively, $PSL_2(\mathbb{Z})/\tilde{H}(N)$), and therefore its Picard--Vessiot extension is an abelian extension of the function field of $X(N)$ (respectively, $X_H(N)$).
  \end{proof}

  \subsection*{An equation for $X(N)$, for even $N\ge 4$}

  We now turn our attention to the specific case of the modular curve $X(N)$, for even $N\ge 4$.
  
  The subgroup $H(N)$ is contained in the principal congruence subgroup $\Gamma(N)$. Both are normal subgroups of $SL_2(\mathbb{Z})$. Therefore, the quotient $\tilde{\Gamma}(N)/\tilde{H}(N)$ is a normal subgroup of $PSL_2(\mathbb{Z})/H(N)$. Hence, the quotient $X_H(N)=\mathbb{H}/H(N)$ is a Galois branched cover of $X(N)=\mathbb{H}/\Gamma(N)$ with Galois group isomorphic to the finite quotient $\tilde{\Gamma}(N)/\tilde{H}(N)$. This finite quotient is isomorphic to the Klein four-group $V_4\simeq\mathbb{Z}/2\mathbb{Z}\times\mathbb{Z}/2\mathbb{Z}$, and is generated by the classes of the matrices $\begin{pmatrix}1 & N\\0 &1\end{pmatrix}$ and $\begin{pmatrix}1 & 0\\N & 1\end{pmatrix}$ in $SL_2(\mathbb{Z})$.
  
  Moreover, as $\tilde{\Gamma}(N)/\tilde{H}(N)$ is a normal subgroup of $PSL_2(\mathbb{Z})/\tilde{H}(N)$, in this construction $X(N)$ is a normal covering of $X(1)$ with Galois group isomorphic to the finite quotient $PSL_2(\mathbb{Z})/\tilde{\Gamma}(N)$. Therefore, a model for $X(N)$ can be obtained from the model $C_N$ of $X_H(N)$ by taking the quotient by the action of $V_4$. This model then sits inside the quotient space $\mathbb{P}^{M-1}(\mathbb{C})/V_4$, for $M=N/2+1$, which consequently inherits an action of the quotient $PSL_2(\mathbb{Z})/\tilde{\Gamma}(N)$. Finally, by applying Theorem \ref{theo0}, we obtain a linear ordinary differential equation with rational coefficients whose Schwarz map parametrizes the model $C_N/V_4$ of $X(N)$ and whose projective monodromy group is isomorphic to the finite quotient $PSL_2(\mathbb{Z})/\tilde{\Gamma}(N)$.

\subsection{The equation for $X(9)$}

  We conclude by illustrating our main theorem, Theorem \ref{mainthm}, with the concrete case of the modular curve $X(9)$. In this case, we have $M=4$, and therefore the canonical invariant projective model $C_9$ of $X(9)$ is a curve in $\mathbb{P}^3(\mathbb{C})$. The group $PSL_2(\mathbb{Z}/9\mathbb{Z})$ has order $216$, and its minimal central extension $G=PSL_2(\mathbb{Z}/9\mathbb{Z})^{SL}$ has order $648$. The matrices generating this central extension $G$ can be obtained by normalizing the matrices $S_*$ and $T_*$ introduced in Section \ref{thetasec}. 
  
  In this case, the model $C_9$ is isomorphic to $X(9)$, and the equations defining $C_9$ in $\mathbb{P}^3(\mathbb{C})$ are given by the vanishing of the following homogeneous polynomials:
  \begin{align*}
    E_1 &= \zeta^5X_2X_3^2 - \zeta X_0 X_2^2 + X_0^2 X_3, \\
    E_2 &= \zeta^5X_0^2X_2 - \zeta^3X_1^3 - \zeta X_2^2X_3 + X_0X_3^2 + X_1^3,
  \end{align*}
  where $\zeta$ is a primitive $18$-th root of unity. These two equations can be derived from the residue theorem, as discussed in Subsection \ref{modelsec}. The Molien series for the action of $G$ on $\mathbb{C}[X_0,X_1,X_2,X_3]$ is represented by the rational function
  \[
  M_G(t)=\dfrac{1+2t^{12}+2t^{18}+2t^{24}+t^{36}}{(1-t^4)(1-t^6)(1-t^{12})(1-t^{18})}.
  \]
  Therefore, according to the Hironaka decomposition, the algebra of invariants $\mathbb{C}[X_0,X_1,X_2,X_3]^G$ is a finite module over the ring generated by four algebraically independent homogeneous polynomials $F_4$, $F_6$, $F_{12}$, and $F_{18}$ of degrees $4$, $6$, $12$, and $18$, respectively. As a module, it is generated by $1$, two homogeneous polynomials $S_{12}$ and $T_{12}$ of degree $12$, two homogeneous polynomials $S_{18}$ and $T_{18}$ of degree $18$, two homogeneous polynomials $S_{24}$ and $T_{24}$ of degree $24$, and one homogeneous polynomial $S_{36}$ of degree $36$. These generating invariants, together with their syzygies, can be explicitly obtained using a Computer Algebra System.
  
  Using \textsc{Singular}, we obtain:
  \begin{align*}
    F_{4} &= X_0^3X_1+(1-\zeta^3)X_1X_2^3+(\zeta^5-3)X_0X_1X_2X_3+X_1X_3^3,\\
    F_{6} &= 3X_1^6+(5\zeta^4-5\zeta)X_0^4X_2^2+(2\zeta^4)X_0X_2^5+(2-2\zeta^3)X_0^5X_3\\
     & \quad -20\zeta^3X_0^2X_2^3X_3+20\zeta^2X_0^3X_2X_3^2+(5\zeta^2-5\zeta^5)X_2^4X_3^2\\
     & \qquad +(20\zeta^4-20\zeta)X_0X_2^2X_3^3+(5-5\zeta^3)X_0^2X_3^4+2\zeta^2X_2X_3^5.
  \end{align*}
  The remaining invariants are lengthy and not uniquely determined by a few coefficients; therefore, we omit their explicit expressions here. They can, however, be retrieved from the corresponding author's webpage.
  
  Using Gröbner bases, the ideal defining the quotient curve $C_9/G$ is obtained by computing the elimination ideal of $\langle E_1,E_2\rangle$ intersected with the ring of invariants $\mathbb{C}[X_0,X_1,X_2,X_3]^G$: 
 \begin{align*}
 \langle E_1,E_2\rangle\cap\mathbb{C}[X_0,X_1,X_2,X_3]^G &  = \langle F_6, 319S_{12} - 124F_{12}, 528821S_{18} - F_{18},\\
  & \quad -7337S_{24} + 12710T_{24}, 205T_{12} - 46F_{12},\\
  & \quad 528821T_{18} - F_{18}, 4F_{4}^3 - F_{12},\\
  & \quad 15376F_{12}^2 - 101761S_{24}, -124F_{12}S_{36} + 319S_{24}^2,\\
  & \quad 124F_{12}S_{24} - 319S_{36}\rangle.
 \end{align*}
  From this, we compute the rational functions $f_i$, $s_i$, and $t_i$ in the function field of $X(9)/PSL_2(\mathbb{Z}/9\mathbb{Z})$ that are required to apply the algorithm from Section \ref{schwarzinvsec} in order to obtain the linear ordinary differential equation with rational coefficients whose Schwarz map parametrizes $C_9$ and whose projective monodromy group is isomorphic to $PSL_2(\mathbb{Z}/9\mathbb{Z})$. We set:
\begin{align*}
  & f_{4}(z)=z, \quad f_{6}(z)=0, \quad f_{12}(z)=4z^3, \quad f_{18}(z)=12691704\sqrt{3}z^5,\\
  & s_{12}(z)=\dfrac{496}{319}z^3, \quad s_{18}(z)=24\sqrt{3}z^5, \quad s_{24}(z)=\dfrac{246016}{101761}z^6, \quad s_{36}(z)=\dfrac{122023936}{32461759}z^9,\\
  & t_{12}(z)=\dfrac{184}{205}z^3, \quad t_{18}(z)=24\sqrt{3}z^5, \quad t_{24}(z)=\dfrac{91264}{65395}z^6.
\end{align*}
  The resulting equation is of order $4$ and has coefficients in $\mathbb{C}(z)$, where $z$ is a \emph{Hauptmodul} for $X(9)/PSL_2(\mathbb{Z}/9\mathbb{Z})$ and ramifies at $z=0,1,\infty$. The explicit form of this equation is
\begin{align*}
  0 & = y^{(4)} + \dfrac{7z - 3}{z(z - 1)}y^{(3)} + \dfrac{2213z^2 - 1895z + 162}{216z^2(z - 1)^2}y^{(2)} + \dfrac{26291z - 7425}{11664z^2(z - 1)^2}y^{(1)} \\
   & \qquad - \dfrac{253}{559872z^2(z - 1)^2}y,
\end{align*}
  where the \emph{Hauptmodul} was chosen such that the exponents of the equation at $z=0$ are $0, 1/2, 1, 3/2$ and at $z=1$ are $0, 1/3, 2/3, 1$. In particular, at $z=\infty$ we have $-1/36, 3/36, 11/36, 23/36$. By construction, the projective monodromy group of this equation is isomorphic to $PSL_2(\mathbb{Z}/9\mathbb{Z})$.

\bibliographystyle{plain}
\bibliography{DGT}

\end{document}